\documentclass[a4paper,10pt]{amsart}
\usepackage{latexsym}
\usepackage[thinlines]{easytable}
\usepackage{amsmath,mathtools,amsfonts,amssymb}
\usepackage{mathrsfs}
\usepackage{hyperref}
\usepackage[retainorgcmds]{IEEEtrantools}
\usepackage{enumerate}
\usepackage{enumitem}
\usepackage{amsthm}
\usepackage{pgf}
\usepackage{fullpage}
\usepackage{graphicx}
\usepackage{caption}
\usepackage{color}
\usepackage{float}
\usepackage{qtree}
\usepackage{subfigure}
\usepackage[final]{microtype}
\usepackage{graphicx}
\usepackage{verbatim}
\usepackage{tikz}
\usepackage{tikz-qtree}
\usetikzlibrary{arrows,automata,positioning}

\theoremstyle{definition} \newtheorem{definition}{Definition}[section]
\theoremstyle{plain} \newtheorem{theorem}[definition]{Theorem}
\theoremstyle{plain} \newtheorem{corollary}[definition]{Corollary}
\theoremstyle{plain} \newtheorem{lemma}[definition]{Lemma}
\theoremstyle{definition} \newtheorem{remark}[definition]{Remark}
\theoremstyle{plain}
\newtheorem{proposition}[definition]{Propostion}
\theoremstyle{plain} \newtheorem{claim}[definition]{Claim}
\theoremstyle{plain} \newtheorem{conjecture}[definition]{Conjecture}
\theoremstyle{definition}
\newtheorem{example}[definition]{Example}

\newcommand{\Skip}[3]{#1#2 \ldots#2 #3}

\newcommand{\ltilde}[1]{\overset{\sim}{#1}}

\newcommand{\gen}[1]{\langle #1 \rangle}

\newcommand{\T}[1]{\mathcal{#1}}

\newcommand{\xn}[1]{X_{#1}}
\newcommand{\xnp}[1]{\xn{#1}^{+}}
\newcommand{\xns}[1]{\xn{#1}^{*}}
\newcommand{\xnl}[1]{\xn{n}^{#1}}
\newcommand{\xnn}{X_{n}^{\N}}
\newcommand{\xnz}{X_{n}^{\Z}}

\newcommand{\pn}[1]{\mathcal{P}_{#1}}
\newcommand{\spn}[1]{\widetilde{\T{P}}_{#1}}
\newcommand{\hn}[1]{\mathcal{H}_{#1}}
\newcommand{\shn}[1]{\widetilde{\T{H}}_{#1}}
\newcommand{\wequal}{=_{\omega}}
\newcommand{\nwequal}{\ne_{\omega}}

\newcommand{\core}{\mathrm{Core}}
\newcommand{\coredist}{\mathrm{Core Dist}}

\newcommand{\Z}{\mathbb{Z}}
\newcommand{\N}{\mathbb{N}}

\newcommand{\dual}[1]{#1^{\vee}}

\DeclarePairedDelimiter{\ceil}{\lceil}{\rceil}
\DeclarePairedDelimiter{\floor}{\lfloor}{\rfloor}

\makeatletter
\renewcommand*{\eqref}[1]{%
	\hyperref[{#1}]{\textup{\tagform@{\ref*{#1}}}}%
}
\makeatother

\begin{document}

\author{
  Olukoya, Feyishayo\\
 Department of Mathematics,\\
		University of Aberdeen, \\ 
		Fraser Noble Building,\\ 
		Aberdeen,\\
    \texttt{feyisayo.olukoya@abdn.ac.uk}
}
\title{The core growth of strongly synchronizing transducers}

\markboth{F. Olukoya}
{The core growth of strongly synchronizing transducers}

\begin{abstract}
	We introduce the notion of `core growth rate' for strongly synchronizing transducers. We explore some elementary properties of the core growth rate and give examples of transducers with exponential core growth rate.  We conjecture that all strongly synchronizing transducers which generate an automaton group of infinite order have exponential core growth rate.  There is a connection to the group of automorphisms of the one-sided shift. More specifically, the results of this article are related to the question of whether or not there can exist infinite order automorphisms of the one-sided shift with infinitely many roots.
\end{abstract}


\maketitle

%\ccode{Mathematics Subject Classification 2010: 68Q45}

\section{Introduction}
The growth of groups and semigroups has received a lot of attention since Milnor first posed the question of the existence of a group of intermediate growth and the very first example of such a group was given by Grigorchuk in 1968 \cite{Grigorchuk68}. In particular, as all known examples of groups of intermediate growth are automaton groups (that is groups generated by Mealy-automata), the question of determining the growth rate of a group generated by an automaton is an important one.  

This question has been considered for certain classes of automaton. More specifically, the paper \cite{IKlimann} shows that whenever the automaton group of a \emph{bireversible automaton} has an element of infinite order, then the growth rate of the group is exponential. Whilst the paper  \cite{OlukoyaOrder} shows that in the class of \emph{strongly synchronizing  automaton}, generalizing the class of \emph{reset automaton} as studied in \cite{SilvaSteinberg}, the automaton group generated by such an automaton has exponential growth. The paper \cite{TGodin} considers this question for generic automata.

In this article we introduce a new notion of growth for strongly synchronizing automaton which we call the \emph{core growth rate}. We recall that a mealy automaton is said  to be strongly synchronizing (\cite{BCMNO}) if there is a number $k$, the \emph{synchronizing level}, such that when an input of length $k$ is processed from any state, the resulting state depends only on the input word. The class of strongly synchronizing automaton is closed under automaton products, moreover, the set of states which are reached by all input words longer than the synchronizing level of such an automaton, form a sub-automaton called the \emph{core}. A strongly synchronizing automaton is called \emph{core} if it is equal to its core. 

The core growth rate measures how the core of a strongly synchronizing  automaton grows with powers of the automaton. For an automaton, the growth rate of the group or semigroup it generates is connected to the growth of the number of states, thus in general the core growth rate is less than the growth rate of the group or semigroup generated by a strongly synchronizing automaton. 

There is a connection to the group of automorphisms of the one-sided shift.  In the forthcoming paper \cite{BleakCameronOlukoya1} and also in the author's PhD thesis \cite{Olukoya2018}, it is shown that automorphisms of the one-sided shift of $n$ letters, $n \ge 2$, correspond in a natural way to invertible strongly synchronizing automaton which are core and whose inverses are also strongly synchronizing and core. In particular, it is shown that such automata naturally induce automorphisms of the one-sided shift on $n$-letters. Moreover, the map induced by the core of the resulting automata obtained by taking the automaton product of two core strongly synchronizing automata $A$ and $B$, is equal to the composition of the maps induced by $A$ and by $B$. Thus for each element of the automorphism of the one-sided shift, there is a smallest core strongly synchronizing automaton which represents this automorphism. The core growth rate then  counts the minimal amount of  `combinatorial data'  required to represent powers of this automorphism.  It is an open question in  \cite{BoyleLindRudolph} if there is an element of the automorphism of the two-sided shift of infinite order with an $k$\textsuperscript{th} root for every $k \in \N$, $k \ge 1$. The results of this paper have bearing on the analogous question for the one-sided shift. Since if all elements of infinite order in the group of automorphisms of the one-sided shift on $n$-letters have exponential core growth rate, then there can be no element with a $k$\textsuperscript{th} root for infinitely many natural numbers $k$.

Our main result demonstrates that for every $n \in \N$, $n \ge 2$, there are invertible strongly synchronizing automaton with exponential core growth rate. We also show that under certain conditions, that the core growth rate of a strongly synchronizing automaton is at least polynomial. We conjecture that it is in fact the case that any invertible strongly synchronizing automaton, which generates an infinite group, has core  exponential core growth rate and moreover the size of the $i$th power of the core always exceeds $i$ (this latter condition controls the size of the powers of the core for small $i$).

 The paper is organised as follows: in Section~\ref{preliminaries} we introduce the terminology, notations and results we will make use of in the  article, in Section~\ref{main} we prove the main result as well as explore some elementary properties of the core growth rate.

\section{Preliminaries} \label{preliminaries}
In this section we introduce the terminology and results that we will require.

Throughout this article, $n \in \N$ will be an natural number greater than or equal to $2$, and  $\xn{n} = \{0,1,\ldots, n-1\}$ an alphabet of cardinality $n$. We write $\xnp{n}$ for the set of all finite, non-empty words over the alphabet $\xn{n}$, we write $\epsilon$ for the empty word and $\xns{n}$ for the set $\xnp{n} \sqcup \{ \epsilon\}$. For $k \in \N$, $\xnl{k}$ denotes the set of all words of length exactly $k$. We have a map $| \cdot | : \xns{n} \to \N$ which returns the length of a word. 

We write $\xnn$ and $\xnz$ for, respectively, the set of all right-infinite and bi-infinite words over the alphabet $\xn{n}$. Equipping $\xn{n}$ with the discrete topology, and taking the product topology on $\xnn$ and $\xnz$ makes these spaces homeomorphic Cantor space.

We shall be concerned with groups of homeomorphisms and monoids of continuous maps of the spaces $\xnn$ and $\xnz$ which may be described by finite state machines that we introduce below.

\begin{definition}
	In our context a \emph{transducer} $A$ is a tuple $A = \gen{X_n, Q_{A}, \pi_{A}, \lambda_{A}}$ where:
	
	\begin{enumerate}[label = (\arabic*)]
		\item $X_n$ is both the input and output alphabet.
		\item $Q_{A}$ is the set of states of $A$.
		\item $\pi_A$, the \emph{transition} function, is a map:
		\[
		\pi: X_n\sqcup\{\epsilon\} \times Q_A \to Q_A
		\]
		\item $\lambda_A$, the \emph{output} or \emph{rewrite} function, is a map:
		\[
		\lambda_A: X_n\sqcup\{\epsilon\} \times Q_A \to X_n^{\ast}
		\]
	\end{enumerate}
\end{definition}

We take the convention that $\pi_{A}(\epsilon, q) = q$ for any $q \in Q_A$, and also $\lambda_A(\epsilon, q) = \epsilon$. 
If $|Q_A| < \infty$ then we say the transducer $A$ is finite. The transducer $A$  is called \emph{synchronous} or a \emph{Mealy automaton} if $\lambda_A$ also satisfies, $|\lambda(x,q)| = |x|$ for any $x \in X_n$ and $q \in Q_A$.  We shall only be concerned with synchronous transducers in this work, thus we abbreviate Mealy-automaton to automaton and shall frequently interchange the words transducer and automaton.

Let $q \in Q_A$ be a state, then we say $A$ is \emph{initialised at $q$} if all inputs are processed from the state $q$ and we write $A_{q}$ to denote this. The transducer $A_{q}$ is then called an \emph{initial transducer}.

We extend the domain  of $\pi$ and $\lambda$ to $X_n^{\ast} \times Q_A$ using induction and the rules:

\begin{IEEEeqnarray}{rCl}
	\pi_A( \Gamma x, q) &=& \pi_A(x, \pi_A(\Gamma,q)) \\
	\lambda_A( \Gamma x, q) &=& \lambda_A(\Gamma,q) \lambda_A(x,\pi_A(\Gamma,q)) 
\end{IEEEeqnarray}

where $\gamma \in X_n^{\ast}$, $x \in X_n$  and  $q \in Q_A$. 

For a word in $\Gamma \in X_n^{\ast}$ and states $q, p \in Q_A$, the phrase \emph{ read $\Gamma$ from state $q$ into $p$} or variations of this phrase, means precisely that  $\pi_A(\Gamma, q) = p$. If we additionally say that \emph{ the output is $\Delta$} then we mean  $\Delta = \lambda_A(\Gamma, q)$. 

Each state $q \in Q_A$ induces a continuous map from Cantor space $X_n^{\mathbb{N}}$ to itself. If this map is a homeomorphism then we say that $q$ is a \emph{homeomorphism state}. Two states $q_1$ and $q_2$ are then said to be $\omega$-equivalent if they induce the same continuous map. (This is can be checked in finite time.) A transducer, therefore, is called \emph{minimal} if no two states are $\omega$-equivalent. Two minimal transducers, $A = \gen{X_n, Q_A, \pi_A, \lambda_A}$ and $B= \gen{X_n, Q_B, \pi_B, \lambda_B}$, are said to be $\omega$-equivalent if there is a bijection $f: Q_A \to Q_B$ such that $q$  and $f(q)$ induce the same continuous map for $q \in Q_A$. In the case where $A$ and $B$ are $\omega$-equivalent then we write $A \wequal B$, otherwise we write $A \nwequal B$. We shall be concerned with various groups and monoids whose elements are $\omega$-equivalence classes of transducers, however, for convenience, we shall introduce these objects as though their elements were transducers. In particular, we often do not distinguish between the $\omega$-equivalence class of a transducer and its representative.

A synchronous transducer $A$ is said to be \emph{invertible (in the automaton theoretic sense)} if all the states of $A$ are homeomorphism states. Equivalently, $A$ is invertible if for every state $q \in Q_A$ the map $\lambda(\cdot, q): \xn{n} \to \xn{n}$ is a bijection. In this case the inverse of $A$ is the transducer $A^{-1}:= \gen{ Q_{A}^{-1}, \xn{n}, \pi_{A ^{-1}, \lambda_{A}^{-1}}$ where $Q_{A}^{-1}:= \{ q^{-1} \vert q \in Q_A \}$ and for $x \in \xn{n}$ and $p,q \in Q_{A}$, $\pi_{A}^{-1}(x, q^{-1}) = p^{-1}$ and $\lambda_{A}^{-1}(x, q^{-1}) = y$ if and only if $\pi_{A}(y,q) = p$ and $\lambda_{A}(y,q) = x$. We shall shortly introduce another sense in which a synchronous transducer is invertible.

Given two transducers $A = \gen{X_n, Q_A, \pi_A, \lambda_A}$ and $B= \gen{X_n, Q_B, \pi_B, \lambda_B}$, the product $A*B$  shall be defined in the usual way. The set of states of $A*B$ will be $Q_A \times Q_B$, and the transition and rewrite functions, $\pi_{A*B}$ and $\lambda_{A*B}$ of $A*B$ are defined by the rules:

\begin{IEEEeqnarray}{rCl}
	\pi_{A*B}(x, (p,q)) &=& (\pi_{A}(x,p), \pi_B(\lambda_A(x,p),q)) \\
	\lambda_{A*B}(x,(p,q)) &=& \lambda_{B}(\lambda_A(x,p),q)
\end{IEEEeqnarray}

Where $x \in X_n\sqcup{\epsilon}$, $p \in Q_A$ and $q \in Q_B$. For $I \in \N$, $i \ge 1$, $A^{i} = A_{1} \ast A_{2}\ast \ldots \ast A_{i}$ where $A_j = A$ for all $1\le j \le i$ , and $A^{-i} = (A^{-1})^{i}$. If $A:= \gen{X_n, Q_A, \lambda_A, \pi_{A}}$ then we shall set $A^{i} = \gen{X_n, Q_A^{i}, \lambda_{Ai}, \pi_{Ai}}$. 

We are also be interested in the do called \emph{dual automaton}. Let $A$ be synchronous transducer, then the dual automaton of $A$ is the automaton $\dual{A} = \gen{Q_{A}, \xn{n}, \pi_{\dual{A}}, \lambda_{\dual{A}}}}$ with state set $\xn{n}$, alphabet set $Q_A$, and transition and output function defined by $\pi_{\dual{A}}(q,x) = y$ if and only if $\lambda_{A}(x,q) = y$ and $\lambda_{\dual{A}}(q,x) = p$ if and only if $\pi_{A}(x,q) = p$.

If $A = \gen{X_n, Q_A, \pi_A, \lambda_A}$ is a synchronous transducer, then as each state $q$ of $A$ induces a continuous function of $X_n^{\mathbb{Z}}$ we may consider the subsemigroup (or group in the case that $A$ is invertible) of the endomorphisms of $X_n^{\Z}$ generated by the set $\{A_{q} | q \in Q_{A}\}$. We refer to this semigroup or group as the \emph{automaton semigroup or automaton group generated by $A$}. We also consider the monogenic semigroup $\gen{A} = \{A^i | i \in \mathbb{N}\}$ or cyclic group $\gen{A} = \{A^i | i \in \Z \}$, and call these the \emph{the semigroup or  group generated by $A$}. Whenever there is any ambiguity we shall make it explicit that $\gen{A}$ refers either to  the semigroup or group generated by $A$.

The following definition is from the paper \cite{BCMNO} and it is with these class of transducers this work will be concerned.

\begin{definition}
	Given a non-negative integer $k$ and an automaton $A=\gen{X_n,Q_A,\pi_A,\lambda_A}$, we say that \emph{$A$ is synchronizing at level $k$} or \emph{strongly synchronizing} if there is a map $\mathfrak{s}:X_n^k\to Q$, so that for all $q\in Q_A$ and any word $\Gamma\in X_n^k$ we have $\mathfrak{s}(\Gamma)=\pi(\Gamma,q)$. That is, the location in the automaton is determined by the last $k$ letters read. We call $\mathfrak{s}$ the synchronizing map for $A$, the image of the map $\mathfrak{s}$ the \emph{core of $A$}, and for a given $\Gamma \in X_n^k$, we call $\mathfrak{s}(\Gamma)$ the \emph{state of $A$ forced by $\Gamma$}. If $A$ is invertible, and $A^{-1}$ is strongly synchronizing at some level $0\le l \in \mathbb{N}$, then  we say that \emph{$A$ is bi-synchronizing at level $\max(k,m)$}. If $A$ is strongly synchronizing but not bi-synchronizing then we shall say $A$ is \emph{one-way synchronizing}.
\end{definition}

\begin{remark}
	\begin{enumerate}[label=(\arabic*)]
		\item It is an easy observation that for a strongly synchronizing transducer the core of $A$ is a strongly synchronizing transducer in its own right. We denote this transducer by $\core(A)$. If $A = \core(A)$, then we say that \emph{$A$ is core}.
		\item For a synchronous, strongly synchronizing transducer, $A$, $\core(A)$ induces a shift-commuting map from $X_n^{\mathbb{Z}}$ to itself. This map $f_{A}$ is defined as follows: given a bi-infinite string $(x_i)_{i \in \mathbb{Z}}$, then $(x_i)f_{A} = \pi(x_i,\mathfrak{s}(x_{i-k}\ldots x_{i-1}))$.  
	\end{enumerate}
\end{remark}
The following straight-forward lemma can be found in the paper \cite{OlukoyaOrder}.

\begin{lemma} \label{claim-SyncLengthsAdd}
	Let $A = \gen{X_n,Q_A,\pi_{A},\lambda_{A}}$ and $B = \gen{X_n,Q_{B},\pi_B,\lambda_{B},}$ be synchronous automata synchronizing at levels $j$ and $k$ respectively. Then $A*B$ is synchronizing at level $j+k$.
\end{lemma}

Let $\widetilde{\mathcal{P}}_{n}$ be the set of core, strongly, synchronizing, synchronous transducers. Define a product from $\spn{n}$ to itself by  $(A, B) \mapsto \min{\core(A \ast B)}$ where $\min{\core(A\ast B)}$ is the minimal transducer representing the core of the product of $A$ and $B$. Since the operations of minimising and reducing to the core commute with each other, the order in which we perform these operations is irrelevant.

It is a result in the forthcoming paper \cite{BleakCameronOlukoya1} that the monoid $\spn{n}$ is isomorphic to a submonoid of the monoid of endomorphisms of the shift dynamical system which, together with negative powers of the shift map generates the endomorphisms of the shift dynamical system. 
  
Let $\pn{n}$ be the submonoid of $\spn{n}$ consisting of those elements of $\spn{n}$ which  induce homeomorphisms of $X_n^{\mathbb{Z}}$. The group $\hn{n}$ is the subset of $\pn{n}$ consisting of those transducers $H$ which have an automaton theoretic inverse (this inverse is in fact again in $\pn{n}$).  It is a result in the forthcoming paper \cite{BleakCameronOlukoya1} and also in the author's PhD thesis \cite{Olukoya2018}, that $\hn{n}$ is isomorphic to the group of automorphisms of the one-sided shift on $n$ letters. Finally define $\shn{n}$ to be those elements of $\spn{n}$ which have  an automaton theoretic inverse (note that this inverse is not always again an element of $\spn{n}$).

The paper \cite{OlukoyaOrder} associates to an element $A \in \spn{n}$ finite graphs $G_{r}(A)$ for every $r \in \N$ greater than or equal to the minimal synchronizing level of $A$ whose vertices are subsets of $Q_{A} \times Q_{A}$. For a given  $r \in \N$ greater than or equal to the minimal synchronizing level of $A$, the graph $G_{r}(A)$ is called the \emph{graph of bad pairs of $A$ (at level $r$)}.

The following result is from \cite{OlukoyaOrder}:

\begin{proposition} \label{infinite=expgrowth1}
	Let $A \in \widetilde{\T{H}}_{n}$ and suppose that $A$ is synchronizing at level $k$ and is minimal. Let $G_{j}(A)$ be the graph of bad pairs for some $j \ge k \in \N$. Suppose there is a subset $\T{S}$ of the set of states of $A$, such that the following things hold:
	\begin{enumerate}[label=(\arabic*)]
		\item $|\T{S}| \ge 2$,
		\item the set $\T{S}(2)$ of two element subsets of $\T{S}$ is a subset of the vertices of $G_{j}(A)$,
		\item for each element of $\T{S}(2)$ there is a vertex  accessible from it which belongs to a circuit.
	\end{enumerate}
	Then the automaton semigroup generated by $A$ contains a free semigroup of rank at least $|\T{S}|$. In particular the automaton semigroup generated by $A$ has exponential growth.
\end{proposition}

For each $k \in \N$, there is a homomorphism (\cite{OlukoyaOrder}) from the set $\spn{n}$ to the  monoid of transformations of $\xnl{k}$ defined as follows. For $A \in \spn{n}$ let $\overline{A}_{k}$ be the transformation of $\xnl{k}$ defined by: $x \mapsto y$ where $y$ is given by $\lambda_{A}(x, q_x)$ for $q_x$ the unique state of $A$ for which $\pi_{A}(x, q_x) = q_x$.

\section{Growth rates of the core of elements of \texorpdfstring{$\spn{n}$}{lg}}\label{main}
\begin{definition}[Core growth rate]
 Let $A \in \T{P}_{n}$ be an automaton, and let $\chi$ be one of `logarithmic', `polynomial', and, `exponential', then we say that $A$ has \emph{core $\chi$ growth (rate)} if the core of the minimal representative of powers of $A$ grows at a rate $\chi$ with powers of $A$.
 
\end{definition}

The lemma below, which is from the paper \cite{OlukoyaOrder}, indicates that there are many examples of elements of $\spn{n}$ and $\shn{n}$, $n \in \mathbb{N}$ and $n \ge 2$ which have core exponential growth.

\begin{lemma}\label{core=power}
	Let $A = \gen{X_n,Q,\pi,\lambda} \in \widetilde{\T{P}}_{n}$ be a transducer, which is synchronizing at level $k$. Furthermore assume that for every $\Gamma \in X_n^k$ and for all states $q \in Q$, there is a state $p \in Q$ such that $\lambda(\Gamma,p) \in W_{q}$. Then under this condition, $A$ has the property that for all $m \in \mathbb{N}$, $Core(A^m) = A^m$.
\end{lemma}

 In particular the Cayley machine (see  \cite{SilvaSteinberg} for a definition) of any finite group. Notice moreover that Lemma~\ref{core=power} applies to transducers without homeomorphism states. The transducer in Figure \ref{shiftmaphascoreexpgrowthrate} is a non-minimal strongly synchronizing transducer whose action on $X_2^{\Z}$ induces the shift-homeomorphism. We call this the 2-shift transducer. This transducer satisfies the hypothesis of Lemma~\ref{core=power} and so has core exponential growth rate. 

\begin{figure}[H]
\begin{center}
 \begin{tikzpicture}[shorten >=0.5pt,node distance=3cm,on grid,auto] 
 \tikzstyle{every state}=[fill=none,draw=black,text=black]
    \node[state] (q_0)   {$a_1$}; 
    \node[state] (q_1) [right=of q_0] {$a_2$};  
     \path[->] 
     (q_0) edge [loop left] node [swap] {$0|0$} ()
           edge [bend left]  node  {$1|0$} (q_1)
     (q_1) edge [loop right]  node [swap]  {$1|1$} ()
           edge [bend left]  node {$0|1$} (q_0);
 \end{tikzpicture}
 \end{center}
 \caption{The shift map has core exponential growth rate.}
 \label{shiftmaphascoreexpgrowthrate}
 
\end{figure}
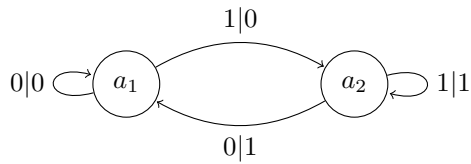

The transducer below, an element of $\shn{2}$, also satisfies Lemma~\ref{core=power}, and so has core exponential growth rate.

\begin{figure}[H]
	\begin{center}
		\begin{tikzpicture}[shorten >=0.5pt,node distance=3cm,on grid,auto] 
		\tikzstyle{every state}=[fill=none,draw=black,text=black]
		\node[state] (q_0)   {$a_1$}; 
		\node[state] (q_1) [right=of q_0] {$a_2$};  
		\path[->] 
		(q_0) edge [loop left] node [swap] {$0|1$} ()
		edge [bend left]  node  {$1|0$} (q_1)
		(q_1) edge [loop right]  node [swap]  {$1|1$} ()
		edge [bend left]  node {$0|0$} (q_0);
		\end{tikzpicture}
	\end{center}
	\caption{An element of $\shn{2}$ with core exponential growth rate.}
	\label{shntwocoreexp}
	
\end{figure}

If we restrict to $\hn{n}$, then it is a result due to Hedlund \cite{Hedlund} that $\hn{2}$ is the cyclic group of order 2. However using Lemma \ref{core=power} one can verify that the element $H$ of   $\hn{n}$, $n \ge 4$, shown in Figure \ref{coreexpgrowthexampleinh4} has core exponential growth rate.

\begin{figure}[H]
\begin{center}
 \begin{tikzpicture}[shorten >=0.5pt,node distance=3cm,on grid,auto] 
 \tikzstyle{every state}=[fill=none,draw=black,text=black]
    \node[state] (q_0)   {$a_1$}; 
    \node[state] (q_1) [right=of q_0] {$a_2$};  
     \path[->] 
     (q_0) edge [in=145,out=115,loop] node [swap] {$0|1,$} node [xshift=-0.2cm, yshift=0.62cm] {$x|x$} ()
           edge [in=180,out=210, loop] node  {$1|2$} ()
           edge [bend left]  node  {$2|0$,$3|3$} (q_1)
     (q_1) edge [out=15,in=45, loop]  node [swap]  {$2|3$} (q_0)
           edge [out=355, in=325,loop] node {$3|0$} ()
           edge [bend left]  node {$0|2$, $1|1$, $x|x$} (q_0);
 \end{tikzpicture}
 \end{center}
 \caption{An element of $\hn{4}$ with core exponential growth rate. Here $x \in \xn{n} \backslash\{0,1,2,3\}$.}
 \label{coreexpgrowthexampleinh4}
 
\end{figure}
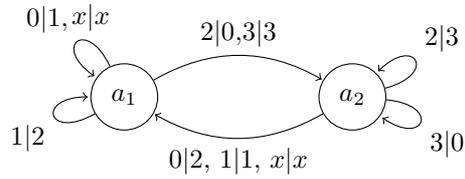

  We have now shown that for $n \ge 4$ $\hn{n}$ contains elements with core exponential growth. This leaves $\hn{3}$.
 
  The transducer $G$ shown in Figure \ref{anelementofH3witcoreexponentialgrowth} is an element of $\hn{3}$, we shall show that this element has core exponential growth rate. Our argument for demonstrating this is somewhat convoluted.
 
\begin{figure}[H]
\begin{center}
 \begin{tikzpicture}[shorten >=0.5pt,node distance=3cm,on grid,auto] 
 \tikzstyle{every state}=[fill=none,draw=black,text=black]
    \node[state] (q_0)   {$b$}; 
    \node[state] (q_1) [right=of q_0] {$a$};  
     \path[->] 
     (q_0) edge [in=145,out=115,loop] node [swap] {$0|0$} ()
           edge [in=180,out=210, loop] node  {$2|1$} ()
           edge  node  {$1|2$} (q_1)
     (q_1) edge [out=130,in=45]  node [swap]  {$0|0$} (q_0)
           edge [loop right] node [swap] {$1|2$} ()
           edge [bend left]  node {$2|1$} (q_0);
 \end{tikzpicture}
 \end{center}
 \caption{An element of $\hn{3}$ with core exponential growth.}
 \label{anelementofH3witcoreexponentialgrowth}
\end{figure}
 
The graph of bad pairs of $G$  at level 1 has a loop it then it follows by Proposition \ref{infinite=expgrowth1} that the automaton semigroup generated by $G$ has exponential growth and is in fact a free semigroup. This means that different words in $\{a,b\}$ of the same length represent inequivalent states of some power of $G$. Since no reductions can be made, we will denote by $Core(G^i)$ the automaton representing the core of $G^i$ for some $i \in \mathbb{N}$.

Observe that $b^i$ is a state of $Core(G^i)$ for all $i \in \mathbb{N}$, since $\pi(0,b) = b$ and $G_b(0) = 0$. Therefore we can treat $G$ as an initial automaton with start state $b$. 

To keep the analysis simple we shall reduce to the case of a Mealy-automaton on a two letter alphabet which will serve as a `dummy' variable for $G$ in a sense that will be made precise. To do this, consider the binary tree in Figure \ref{fig: binary tree of transitions} representing how the initial transducer $G_{bb} := Core(G^2)$ transitions on certain inputs.  The left half of tree corresponds to transitions from the set $\{1\} \times \{0,2\} \times \{0,1\} \times \{0,2\} \times \{0,1\} \ldots$ the right half of the tree corresponds to transitions from the set $\{2\} \times \{0,1\} \times \{0,2\} \times \{0,1\} \times \{0,2\} \ldots$. Let $T_1 := \{0,2\} \times \{0,1\} \times \{0,2\} \times \{0,1\} \ldots$ and $T_2:= \{0,1\} \times \{0,2\} \times \{0,1\} \times \{0,2\} \ldots$.
\begin{figure}[H]
\begin{center}
  \begin{tikzpicture}[every tree node/.style={ }, level distance=1cm,sibling distance=0.25cm, edge from parent path={(\tikzparentnode) -- (\tikzchildnode)}]
  \Tree
  [.\node[draw,circle]{bb};
      \edge node[auto=right,pos=.5] {$1|1$};
      [.\node[draw,circle]{ab}; 
         \edge node[auto=right,pos=.6] {$0|2$};
         [.\node[draw,circle]{ba}; 
             \edge node[midway,left] {$0|1$};
             [.\node[draw,circle]{bb}; 
                 \edge node[midway,left] {$0|0$};
                 [.$\vdots$ ]
                 \edge node[midway,right]{$2|2$};
                 [.$\vdots$ ]
             ]
             \edge node[midway,right]{$1|0$};
             [.\node[draw,circle]{ab};
                 \edge node[midway,left] {$0|2$};
                 [.$\vdots$ ]
                 \edge node[midway,right]{$2|0$};
                 [.$\vdots$ ]             
              ]
         ]
         \edge node[auto=left,pos=.6] {$2|0$};
         [.\node[draw,circle]{bb}; 
              \edge node[midway,left] {$0|0$};
              [.\node[draw,circle]{bb};
                \edge node[midway,left] {$0|0$};
                [.$\vdots$ ] 
                \edge node[midway,right]{$2|2$};
                [.$\vdots$ ] 
               ]
              \edge node[midway,right]{$1|1$};
              [.\node[draw,circle]{ab};
                  \edge node[midway,left] {$0|2$};
                  [.$\vdots$ ]
                  \edge node[midway,right]{$2|0$};
                  [.$\vdots$ ]
               ]
         ]
          ]
      \edge node[auto=left,pos=.5] {$2|2$};
      [.\node[draw,circle]{ba}; 
          \edge node[auto=right,pos=.6] {$0|1$};
          [.\node[draw,circle]{bb}; 
              \edge node[midway,left] {$0|0$};
              [.\node[draw,circle]{bb}; 
                  \edge node[midway,left] {$0|0$};
                  [.$\vdots$ ]
                  \edge node[midway,right] {$1|1$};
                  [.$\vdots$ ]
              ]
              \edge node[midway,right]{$2|2$};
              [.\node[draw,circle]{ba}; 
                  \edge node[midway,left] {$0|1$};
                  [.$\vdots$ ]
                  \edge node[midway,right] {$1|0$};
                  [.$\vdots$ ]
              ]
          ]
          \edge node[auto=left,pos=.6] {$1|0$};
          [.\node[draw,circle]{ab}; 
              \edge node[midway,left] {$0|2$};
              [.\node[draw,circle]{bb}; 
                  \edge node[midway,left]{$0|0$};
                  [.$\vdots$ ]
                  \edge node[midway,right]{$1|1$};
                  [.$\vdots$ ]
              ]
              \edge node[midway,right]{$2|0$};
              [.\node[draw,circle]{ba};
                  \edge node[midway,left]{$0|1$};
                  [.$\vdots$ ]
                  \edge node[midway,right]{$1|0$};
                  [.$\vdots$ ]
               ]
          ]
          ]
  ]
  \end{tikzpicture} 
 \end{center}
 \caption{Binary tree depicting the transitions of $G_{bb}$.} 
 \label{fig: binary tree of transitions}
 \end{figure}
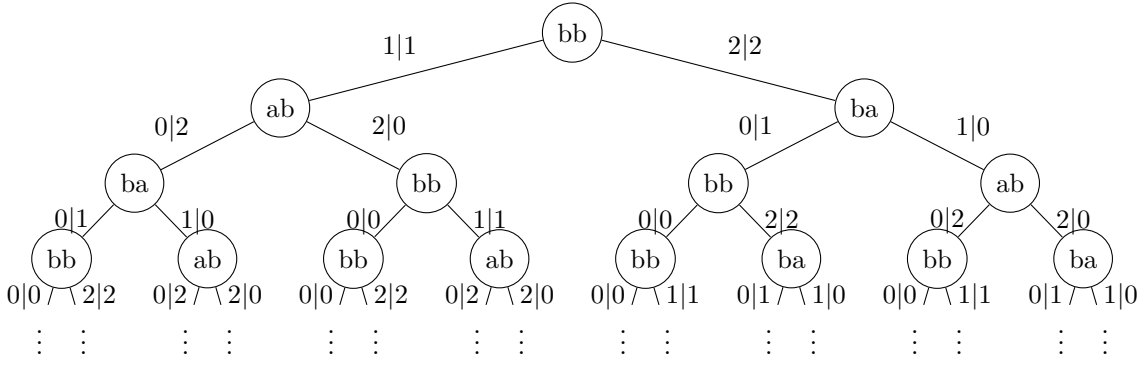
 
  Using Figure \ref{fig: binary tree of transitions} we form a dummy transducer which mimics the transitions of $G_{bb}$ as follows. We shall only be interested in the transitions of this dummy transducer and so whenever we take powers of the dummy transducer we will not minimise it. First form new states $B \sim bb$, $\sigma_1^1 \sim ab$, $\sigma_0^1 \sim ba$, $\sigma_0^0$  and $\sigma_1^0$. Here $\sigma_0^0 $ corresponds to the state $bb$ whenever we read an element of $\{0,1\}$ from $bb$ and $\sigma_1^0$ corresponds to the state $bb$ whenever we read an element of $\{0,2\}$ from $bb$. Now notice that all states on the left half of below the root, at odd levels map $\{0,2\}$ into $\{0,2\}$ and at all states at even levels map $\{0,1\}$ into $\{0,1\}$. Analogously all states on the right half of the tree below the root map $\{0,1\}$ into $\{0,1\}$ at odd levels and $\{0,2\}$ into $\{0,2\}$ at even levels. Since we only care about transitions we may transform the tree into a binary tree by replacing all the 2's with 1's so long as we still encode the information about which side of the tree we are on, and about parity, even or odd, of the level of the tree we are acting on. This is achieved by the states $\sigma_0^0$ and $\sigma_1^0$ which represent the occurrence of $bb$ on the left half of the tree at  even levels and on the right half of the tree at odd levels. The resulting initial transducer $\widetilde{G}_{B} = \gen{\{0,1\},\widetilde{\pi},\widetilde{\lambda}}$ on a two-letter alphabet now transitions similarly to $G_{bb}$, and has states corresponding to  states of $G$. In particular, by construction, any state of $\widetilde{G}^i$ (we do not minimise this transducer as we are interested only in transitions) accessible from $B$ (in $\widetilde{G})$ will correspond to a state in $G^i$ (where we replace $\sigma_i^j$, $i,j = 0,1 $ by the corresponding state of $G$) accessible from $bb$ (in $G$) by reading either a $1$ or $2$ then, in the first case alternating between reading an element of $\{0,2\}$ and an element of  $\{0,1\}$ and in the second between an element of $\{0,1\}$ and an element of $\{0,2\}$.
 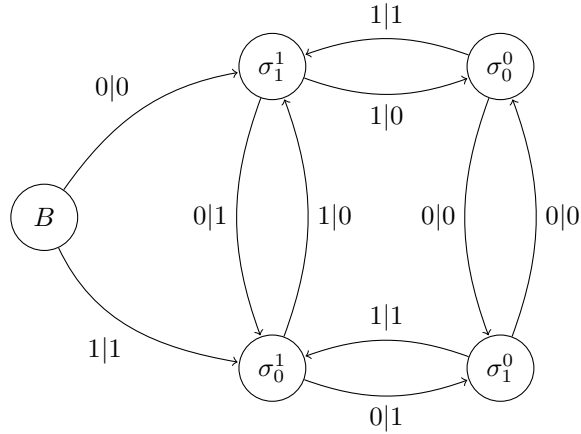
\begin{figure}[H]
 \begin{center}
 \begin{tikzpicture}[shorten >=0.5pt,node distance=4cm,on grid,auto] 
    \node[state] (q_1) [xshift=-1cm,yshift=0cm]   {$B$}; 
    \node[state] (q_2) [xshift =2cm, yshift = -2cm] {$\sigma_0^1$};
    \node[state] (q_3) [xshift =2cm,yshift=2cm] {$\sigma_1^1$};
    \node[state] (q_4) [xshift =5cm,yshift=2cm] {$\sigma_0^0$};
    \node[state] (q_5) [xshift =5cm,yshift=-2cm] {$\sigma_1^0$};
     \path[->] 
     (q_1) edge [out=55, in=190]  node {$0|0$} (q_3)
           edge [out=295, in=170] node[swap] {$1|1$} (q_2)
           
     (q_2) edge [in=290, out =70]  node[swap] {$1|0$} (q_3)
           edge [out=340,in=200] node[swap] {$0|1$} (q_5)
     (q_3) edge [in=110, out=250] node[swap] {$0|1$} (q_2)
           edge [out=340,in=200] node[swap] {$1|0$} (q_4)
     (q_4) edge [out=160,in=20]  node[swap] {$1|1$} (q_3)
           edge [in=110, out=250] node[swap] {$0|0$} (q_5)
     (q_5) edge [out=160,in=20]  node[swap] {$1|1$} (q_2)
           edge [in=290, out =70] node[swap] {$0|0$} (q_4); 
           
 \end{tikzpicture}
 \end{center}
 \caption{The dummy transducer $\widetilde{G}_{B}$.}
 \label{fig: dummy transducer}
 \end{figure}
The point of building the transducer $\widetilde{G}$ is that it encodes the transitions of $G$ in a fashion which is much easier to describe. One should think of $\widetilde{G}$ as a dummy transducer for $G$ in which it is much easier to read transitions as we shall see.
 
 Since we transition  from $B^i$ to $(\sigma_1^1)^i$ by reading $0$ it  suffices to show that the initial transducer $\widetilde{G}_{\sigma_{1}^{1}}$ has exponential growth. Recall that here we are interested in how the number of states  of $\widetilde{G}$ grow without considering the $\omega$-equivalence of these states. We shall then argue from this fact that $G$ has core exponential growth  since the automaton semigroup generated by $G$ is free and the states of $G$ correspond nicely to the states of powers of $\widetilde{G}$ (without minimising). 
 
 First we argue that the number of states of $\widetilde{G}^i_{(\sigma_{1}^{1})^i}$ is at least $2^{\ceil{{i/2}}}$. We stress once more that we are not concerned with the $\omega$-equivalence of some of these states, they merely act as dummy variables for the states of $G_{(bb)^i}^{i}$. In particular whenever we raise $\widetilde{G}_{(\sigma_{1}^{1})}$ to some power, we shall not  minimise it.
 
 Notice that for $x,i,j = 0,1$, 
 
 \begin{IEEEeqnarray}{rCl}
 \widetilde{\pi}(x,\sigma_i^j) &=& \sigma_{i+1}^{x+ij} \label{transition}\\
 \widetilde{\lambda}(x, \sigma_i^j) &=& x+j \mod{2}\label{output}.
 \end{IEEEeqnarray}

 In $\eqref{transition}$ and $\eqref{output}$ subscripts and exponents are taken modulo 2. Since $ a + b \mod{2} = ((a\mod{2}) + (b\mod{2}))\mod{2}$ and $ab\mod{2} = ((a\mod{2})(b\mod{2}))\mod{2}$, we can iterate the above formulae.
 
 We shall require the following notation in order to simplify the discussion that follows. Set, for $i,j \in \mathbb{Z}$, $i\ge 1$.
 
 \begin{equation*}
   \Sigma(i,j):= \sum_{l_1=1}^{j}\sum_{l_2=1}^{l_1}\ldots\sum_{l_i=1}^{l_{i-1}} l_i
 \end{equation*}
 
 If $j=0$ or is negative then take $\Sigma(i,j) = 0$. Notice that the $\Sigma(1,j)$ is simply the sum of the first $j$ numbers $j\ge 1$. Furthermore observe that 
 
 \begin{equation}\label{sumrules}
   \sum_{k=1}^{j} \Sigma(i,k) = \Sigma(i+1, j)
 \end{equation}

\begin{remark}
It is straight-forward to show either by finite calculus or by induction making use of the identity $\sum_{k=m}^{j} \binom{k}{m} = \binom{j+1}{m+1}$ that $\Sigma(i,j) = \binom{j+i}{i+1}$. We shall not require this fact. 
\end{remark} 
 Freeing the symbol $k$, let $k \ge 1 \in \mathbb{N}$ and let $x_1\Skip{x_2}{}{x_k} \in \{0,1\}^k$. In what follows below whenever we have an $x_{i}$ for $ i \in \Z$ and $i < 0$, we shall take $x_i$ to be $0$ and $x_0 = 1$. We have the following claim:
 
 \begin{claim}
 For $i$ even and bigger than or equal to 1 after reading the first $i$ terms the $j$\textsuperscript{th} term of the active state is  
  \begin{equation}\label{formula: i even}
   \sigma_1^{x_i+(j-1)x_{i-1}+ \Sigma(1,j-1)x_{i-2}+  \Sigma(2, j-2)x_{i-3}+ \Sigma(3, j-2)x_{i-4}+ \ldots + \Sigma(i-2, j-i/2)x_1+ \Sigma(i-1, j-i/2)\cdot 1}.
  \end{equation} 
 After reading the first $i+1$ terms of the sequence $x_1\ldots x_k$ through $\widetilde{G}^{k}_{(\sigma_{1}^{1})^k}$ the $j$\textsuperscript{th} term of the active state is 
  \begin{equation}\label{formula: i odd}
  \sigma_0^{x_{i+1}+jx_{i}+\Sigma(1,j-1)x_{i-1}+\Sigma(2,j-1)x_{i-2}+ \Sigma(3,j-2)x_{i-3}+\Sigma(4,j-2)x_{i-4}+ \ldots+\Sigma(i-1,j-i/2)x_1+ \Sigma(i,j-i/2)\cdot 1}
  \end{equation}
 exponents are taken modulo 2.
 \end{claim}
 
 \begin{proof}
 The proof follows by induction and a mechanical calculation making use of $\eqref{sumrules}$, $\eqref{output}$.
 
 We first establish the base cases $i=1$ and $i = 2$. The top row of array \eqref{array: reading 1} consists of $k$ copies of the state $\sigma_1^1$ of $\widetilde{G}_{B}$. The first column of the second row indicates the we are reading the letter $x_1$ through state  $\sigma_1 ^1$. In the second column, the symbol $x_1 +1$ is the input to be read through the second copy of $\sigma_1^1$, and $\sigma_0^{1+x_1}$  is equal to $\widetilde{\pi}( x_1, \sigma_1^1 )$. The remaining columns are to be read in a similar fashion. 
 
 \begin{IEEEeqnarray}{cccccc} \label{array: reading 1}
     & \sigma_1^1 & \sigma_1^1& \sigma_1^1 & \ldots &  \ \sigma_1^1 \nonumber \\
 x_1 \quad & x_1 +1 \  \sigma_0^{1+ x_1} \quad & x_1 + 2 \ \sigma_0^{x_1 +2} \quad& x_1 + 3 \ \sigma_0^{x_1 + 3} \quad  & \ldots& \  x_1 + k \ \sigma_0^{x_1+ k}  
 \end{IEEEeqnarray}
 
 Therefore after reading $x_1$ from the state $(\sigma_1^1)^{k}$ the active state of the transducer $\widetilde{G}^{k}_{(\sigma_1)^{k}}$  is 
 $$ \sigma_{0}^{x_1 +1} \sigma_{0}^{x_1 +2} \sigma_{0}^{x_1 +3}\ldots \sigma_{0}^{x_1 +k} $$
 
 which is as indicated by the formula  \eqref{formula: i odd}.
 
 Now we read $x_2$ through the active state $ \sigma_{0}^{x_1 +1} \sigma_{0}^{x_1 +2} \sigma_{0}^{x_1 +3}\ldots \sigma_{0}^{x_1 +k} $ to establish the case $i = 2$. We shall make use of an array as in \eqref{array: reading 1} to do demonstrate this.
 
 \begin{IEEEeqnarray}{cccc} \label{array: reading 2}
      & \sigma_0^{x_1 +1} & \sigma_0^{x_1 +2}& \sigma_0^{x_1 +3}  \nonumber \\
  x_2 \quad & x_2+ x_1+ \Sigma(1,1) \  \sigma_1^{x_2} \quad & x_2 + 2 x_1 + \Sigma(1,2) \ \sigma_1^{x_2+ x_1+ \Sigma(1,1)} \quad& x_2 + 3 x_1 + \Sigma(1,3)  \ \sigma_1^{x_2+ 2x_1+ \Sigma(1,2)} \quad  
  \end{IEEEeqnarray}
  
  A simple induction shows that the $(k+1)$\textsuperscript{st} entry of the second row is: $$x_2 + k x_1 + \Sigma(1, k) \ \sigma_1^{x_2 + (k-1) x_1 + \Sigma(1, k-1)}$$
  
  and so all the terms of the active state are as indicated by the formula \eqref{formula: i even}.
  
  Now assume that  $i$ is even and $2 \le i \le k-1$ and that the $j$\textsuperscript{th} of the active state after reading the first $i$ terms of $x_1 \ldots x_k$ is as given by the formula \eqref{formula: i even}. We now show that after reading $x_{i+1}$ through the active state the $j$\textsuperscript{th} term of the active state is as given in \eqref{formula: i odd}. We shall proceed by induction on $j$.
  
  By assumption the first term of the active state is
   $\sigma_1^{x_i}$. Therefore $\widetilde{\pi}(x_{i+1},\sigma_1^{x_i}) = \sigma_{0}^{x_{i+1} +  x_i}$ and $\widetilde{\lambda}(x_{i+1},\sigma_1^{x_i}) = x_{i+1} + x_{i}$. Therefore the first term $\sigma_{0}^{x_{i+1} +  x_i}$ of the new active state satisfies the formula \eqref{formula: i odd} with $j = 1$.
   
   By assumption the second term of the current active state is $\sigma_1^{x_i + x_{i-1}} + \Sigma(1,1)x_{i-2}$. Therefore $$\widetilde{\pi}\left(x_{i+1} + x_{i},\sigma_1^{x_i + x_{i-1} + \Sigma(1,1)x_{i-2}}\right) = \sigma_0^{ x_{i+1} + 2 x_{i} + x_{i-1} + \Sigma(1,1) x_{i-2}}$$ and $$\widetilde{\lambda}\left(x_{i+1} + x_{i},\sigma_1^{x_i + x_{i-1} + \Sigma(1,1)x_{i-2}}\right) = x_{i+1} + 2 x_{i} + x_{i-1} + \Sigma(1,1) x_{i-2}.$$
   
   Now we may rewrite  $x_{i+1} + 2 x_{i} + x_{i-1} + \Sigma(1,1) x_{i-2}$ as $x_{i+1} + x x_{i} + \Sigma(1,1) x_{i-1} + \Sigma(2,1) x_{i-2}$ since $\Sigma(2,1) = \Sigma(1,1)$ and $\Sigma(1,1) = 1$. Therefore the 2nd term of the new active state $\sigma_0^{ x_{i+1} + 2 x_{i} + \Sigma(1,1)x_{i-1} + \Sigma(2,1) x_{i-2}}$ satisfies the formula \eqref{formula: i odd} with $j =2$.
   
   Now assume that for $2 \le  j \le k$ the $j-1$\textsuperscript{st} term of the new active state is given by: $$\sigma_0^{x_{i+1}+(j-1)x_{i}+\Sigma(1,j-2)x_{i-1}+\Sigma(2,j-2)x_{i-2}+ \Sigma(3,j-3)x_{i-3}+\Sigma(4,j-3)x_{i-4}+ \ldots+\Sigma(i-1,j-1-i/2)x_1+ \Sigma(i,j-1-i/2)\cdot 1}$$
   
   and the output when $x_i$ is  read through the first $j-1$ terms of the current active state is 
   \begin{IEEEeqnarray*}{rCl}
   x_{i+1}&+&(j-1)x_{i}+\Sigma(1,j-2)x_{i-1}+\Sigma(2,j-2)x_{i-2}+ \Sigma(3,j-3)x_{i-3}+\Sigma(4,j-3)x_{i-4}+ \ldots\\ &+&\Sigma(i-1,j-1-i/2)x_1+ \Sigma(i,j-1-i/2)\cdot 1.
   \end{IEEEeqnarray*}
    Therefore the $j$\textsuperscript{th} term of the new active state will be the active state after  $x_{i+1}+(j-1)x_{i}+\Sigma(1,j-2)x_{i-1}+\Sigma(2,j-2)x_{i-2}+ \Sigma(2,j-3)x_{i-3}+\Sigma(3,j-3)x_{i-4}+ \ldots+\Sigma(i-1,j-1-i/2)x_1+ \Sigma(i,j-1-i/2)\cdot 1$ is read from the current active state. By assumption the current active state is:
    \begin{equation*}
       \sigma_1^{x_i+(j-1)x_{i-1}+ \Sigma(1,j-1)x_{i-2}+  \Sigma(2, j-2)x_{i-3}+ \Sigma(3, j-2)x_{i-4}+ \ldots + \Sigma(i-2, j-i/2)x_1+ \Sigma(i-1, j-i/2)\cdot 1}.
      \end{equation*}
    Making use of the \eqref{transition} and \eqref{output} and the rule \eqref{sumrules}, the new active state is given by
    
    \begin{equation*}
      \sigma_{0}^{x_{i+1} + jx_i + \Sigma(1, j-1) x_{i-1} + \Sigma(2, j-1) x_{i-2} + \Sigma(3, j-2) x_{i-3} + \Sigma(4, j-2) x_{i-4} + \ldots + \Sigma(i-1, j- i/2) x_1 + \Sigma(i, j - i/2) \cdot 1  }
    \end{equation*}
    
    which is exactly the formula given in \eqref{formula: i odd}. 
    
    The case where $i$ is odd  is proved in an analogous fashion.  
 \end{proof}
 
 Observe that for all $i >0$ we have $\Sigma(i,1) = 1$. Now for $i$ even  and $j = i/2 +1$ consider $\widetilde{G}^{i+1}_{(\sigma_{1}^{1})^{i+1}}$, the following formulas determine the exponents of the first $j$ terms of the active state after reading the first $i+1$ terms of the sequence  $x_1, \ldots x_k$. The subscripts of these states are all $0$. 
 
 \begin{IEEEeqnarray*}{rCl}
   & x_{i+1}& +x_i \nonumber \\
   & x_{i+1} & + 2x_i + \Sigma(1,1)x_{i-1} + \Sigma(2,1)x_{i-2} \nonumber \\
   & x_{i+1} & + 3x_{i} + \Sigma(1,2) x_{i-1} + \Sigma(2,2)x_{i-2} + \Sigma(3,1)x_{i-3} + \Sigma(4,1)x_{i-4} \nonumber \\
   &\vdots& \nonumber \\
   & x_{i+1}&+(j-1)x_{i}+\Sigma(1,j-2)x_{i-1}+\Sigma(2,j-2)x_{i-2}+ 
   +\Sigma(i-3,j-i/2+1)x_3+ \Sigma(i-2,j-i/2+1)x_2    \nonumber \\
   & x_{i+1}&+jx_{i}+\Sigma(1,j-1)x_{i-1}+\Sigma(2,j-1)x_{i-2}+ \ldots\Sigma(i-1,j-i/2)x_1+ \Sigma(i,j-i/2)    
 \end{IEEEeqnarray*}
 
Let $y_1,\ldots y_j$ in $\{0,1\}^{j}$  be any sequence. Since the coefficients of the last two terms of all the equations above is $1$, there is a choice of $x_1 \ldots x_{i+1}$ such that the exponent of the $l$\textsuperscript{th} term ($1\le l \le j$) of the active state after reading $x_1 \ldots x_{i+1}$ in $\widetilde{G}^{i+1}_{(\sigma_1^1)^{i+1}}$ is $y_l$. This is achieved inductively, first we solve $x_{i+1} + x_{i} = y_1$ in $\Z_{2}$. This determines $x_{i+1}$ and $x_{i}$. Next we pick  $x_{i-1}$ so that $x_{i+1}  + 2x_i + \Sigma(1,1)x_{i-1} = 0 \mod 2$, and set $x_{i-2} = y_2$.  This determines $x_{i-1}$ and $x_{i-2}$. Therefore we may now pick $x_{i-3}$ so that $x_{i+1}  + 3x_{i} + \Sigma(1,2) x_{i-1} + \Sigma(2,2)x_{i-2} + \Sigma(3,1)x_{i-3} = 0$ and set $x_{i-4} = y_3$. We carry on in this way until we have determined $x_l$ for $i+1 \le l \le 2$. Then we solve  the equation
     $$x_{i+1}+jx_{i}+\Sigma(1,j-1)x_{i-1}+\Sigma(2,j-1)x_{i-2}+ \ldots+ (\Sigma(i,j-i/2) - y_{j}) + \Sigma(i-1,j-i/2)x_1   = 0$$
for $x_1$ in  $\Z_{2}$.

   That is for any  sequence $y_1 \ldots y_j \in  \{0,1\}^{j}$, there is a state of $\widetilde{G}^{i+1}_{(\sigma_1^1)^{i+1}}$ whose first $j$ terms are $\sigma_0^{y_1} \ldots\sigma_0^{y_j}$.
   
Now for $\widetilde{G}_{\sigma_1^1}^{i}$, a similar argument shows for any such sequence $y_1 \ldots y_j$, there is a state of $\widetilde{G}_{(\sigma_1^1)^i}$ whose first $j$ terms is $\sigma_1^{y_1} \ldots\sigma_1^{y_j}$.

 Now using the correspondence stated above that $\sigma_1^1 \sim ab$, $\sigma_0^1 \sim ba$ and $\sigma_0^0 \sim bb$ and $\sigma_1^0 \sim bb$, the states of $\widetilde{G}_{\sigma_1^1}^{i+1}$ and $\widetilde{G}_{\sigma_1^1}^{i}$ correspond to states of $G^{i}$ and $G^{i+1}$ accessible from the state $(bb)^{i}$ and $(bb)^{i+1} $.   Since the automaton semigroup generated by $G$ is free, then two different words in $\{a,b\}^{2(i+1)}$ will correspond to distinct states of $G^{(2(i+1))}$. Now by the arguments above we have that for every element $y_1\ldots y_j$ in the set $\{0,1\}^{j}$ $\widetilde{G}_{\sigma_1^1}^{i+1}$ and $\widetilde{G}_{\sigma_1^1}^{i}$ have states beginning with $\sigma_0^{y_1}\ldots \sigma_{0}^{y_j}$ and $\sigma_1^{y_1}\ldots \sigma_{1}^{y_j}$  respectively. Now using the fact that the automaton semigroup generated by $G$ is free, it follows that for $y_1 \ldots y_j$ and $y_1' \ldots y_j'$ in $\{0,1\}^{j}$, the states $\sigma_{l}^{y_1}\ldots  \sigma_{l}^{y_j}$ and $\sigma_{l}^{y_1'}\ldots  \sigma_{l}^{y_j'}$ for $l \in \{0,1\}$ correspond to distinct states of $G^{j}$.  Therefore $G^{i+1}_{(bb)^{i+1}}$ has at least $2^{i/2 \ +1} = 2^{\ceil{{(i+1)/2}}}$ states  $G^{i}_{(bb)^{i}}$ has at least $2^{\ceil{{(i+1)/2}}}$ states. It now follows that for arbitrary $i \in \mathbb{N}$,  $G_{b^i}$ has at least $2^{\floor{i/2}}$ states for any $i \ge 1 \in \mathbb{N}$. 

The above all together now means that $G$ is an element of $\hn{3}$ with core exponential growth. Therefore we have: 

\begin{theorem}
For any $n > 2$ there are elements of $\hn{n}$ which have core exponential growth. 
\end{theorem}

\begin{remark}
 for $i \in \mathbb{N}$, the maximum difference in the size of elements of $\mathcal{H}_{n}$ which are bi-synchronizing at level $i$ grows exponentially with $i$.
\end{remark}
\begin{proof}
For each $i \in \mathbb{N}$ it is possible to construct an  element of $\hn{n}$ which is bi-synchronizing at level $i$, see Figure \ref{Anelementofh3bisynchronizingatleveli} for an indication of how to do so. On the other hand  there are elements of $\hn{n}$ which are bi-synchronizing at level 1, and which have core exponential growth (for instance the example in Figure \ref{anelementofH3witcoreexponentialgrowth} ). Let $G$ be such an element. Then $\min\core(G^{i})$ is bi-synchronizing at level $i$ by Lemma~\ref{claim-SyncLengthsAdd} and has at least $e^{ci}$  states for some positive constant $c$. Therefore the maximum difference in the size of elements of $\hn{n}$ which are bi-synchronizing at level $i$ is at least $e^{ci} - i -1$.
\end{proof}

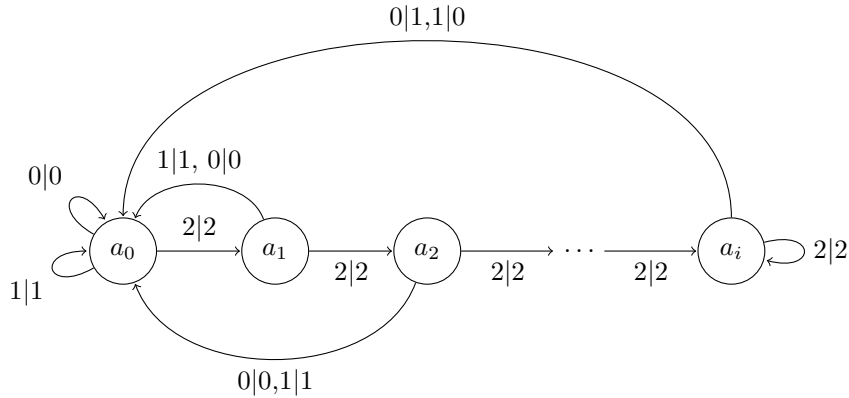
\begin{figure}[H]
\caption{An element of $\hn{3}$ bi-synchronizing at level $i$.}
 \label{Anelementofh3bisynchronizingatleveli}
\begin{center}
 \begin{tikzpicture}[shorten >=0.5pt,node distance=4cm,on grid,auto] 
    \node[state] (q_1) [xshift=0cm,yshift=0cm]   {$a_0$}; 
    \node[state] (q_2) [xshift =2cm, yshift = 0cm] {$a_1$};
    \node[state] (q_3) [xshift =4cm,yshift=0cm] {$a_2$};
    \node (q_4) [xshift =6cm,yshift=0cm] {$\ldots$};
    \node[state] (q_5) [xshift =8cm,yshift=0cm] {$a_i$};
     \path[->] 
     (q_1) edge [out=150, in=120, loop]  node {$0|0$} ()
           edge [out=210, in=180,loop] node {$1|1$} ()
           edge node {$2|2$} (q_2)
     (q_2) edge [in=70, out =110]  node[swap] {$1|1$, $0|0$} (q_1)
           edge  node[swap] {$2|2$} (q_3)
     (q_3) edge [in=290, out=250] node {$0|0$,$1|1$} (q_1)
           edge  node[swap] {$2|2$} (q_4)
     (q_4) edge  node[swap] {$2|2$} (q_5)
     (q_5) edge [out=90,in=90]  node[swap] {$0|1$,$1|0$} (q_1)
           edge [loop right] node[swap] {$2|2$} (); 
           
 \end{tikzpicture}
 \end{center}
\end{figure}
In the subsequent discussion we explore some of the elementary properties of the core growth rate, and state a conjecture about the core growth rates of elements of $\shn{n}$ which have infinite order.

\begin{lemma}
Let $A \in \T{P}_{n}$ be an element of infinite order. If $B$ is conjugate to $A$ in $\T{P}_{n}$ then core growth rate of $B$ is equivalent to the core growth rate of $A$.
\end{lemma}
\begin{proof}
Let $C \in \T{P}_{n}$ be such that $B$ is the minimal transducer representing the core of $C^{-1}AC$. 

Since $\T{P}_{n}$  restricting to the core is a part of multiplication in $\T{P}_{n}$. It follows that $\min(Core(C^{-1}A^{m}C)) \wequal B^{m}$, where $A^m$ and $B^m$ are here identified with the minimal automaton representing the core of $A^{m}$ and $B^{m}$ respectively. 

This readily implies:

\[
  |C||A^m||C| = |B^m|
\]

as required.
\end{proof}

The next lemma shows that the core growth rate is invariant under taking powers.

\begin{lemma} \label{coregrowthratepowerinvariant}
Let $A \in \T{P}_{n}$ of infinite order, and let $\chi$ be one of `exponential', `polynomial', or `logarithmic'. Then if there is some $m \in \mathbb{N}$ such that $Core(A^m)$ has core $\chi$ growth rate, then $A$ also has core $\chi$ growth rate.
\end{lemma}
\begin{proof}
This is a straight-forward observation. 
Let $m \in \mathbb{N}$ be fixed such that $\core(A^m)$ has exponential growth. 

Let $i \in \mathbb{Z}_{m}$ and let $k \in \mathbb{N}$. Now notice that $|\min (A^{km+i})| \ge |\min Core(A^m)^{k+1}|/|A|^{m-i} \ge \exp^{c(k+1)}/|A|^{m-i} \ge \exp^{c(km+i)/m}/|A|^{m-i}$, for a positive constant $c$. Now as every positive integer can be written at some $qm+i$, $0 \le q \in \mathbb{Z}$ and $i \in \mathbb{Z}_{m}$ we are done.

If $\core(A^m)$ has polynomial growth rate, then there are positive numbers $C$ and $d$ such that  $|\min\core(\core(A^m)^{k})| \le Cn^d$. Now consider the following inequalities:
\[
|A|^{i}C(km+i)^d \ge |A|^{i}Ck^d \ge |A|^{i}|\min\core(\core(A^m)^{k})| \ge |A^i \min\core(\core(A^m)^{k})| \ge  |\min\core(A^{km+i})|
\]

An Analogous argument shows that if $\core(A^m)$ has core logarithmic growth rate then so does $A$.
\end{proof}

As a corollary of the lemma above we are able to reduce the question of determining the  core growth rates for non-initial automata to the question of determining the growth rate of initial automata.
 
 \begin{corollary}\label{Reduction to initial automata}
 Let $A \in \shn{n}$ then the core growth rate of $A$ is equivalent to the growth rate of some initial automaton $B_{q_0}$. 
 \end{corollary}
 \begin{proof}
 Consider the transformation $\overline{A}_{1}$ of $\xn{n}$. Observe that there is an $i \in \mathbb{N}$ and  $x \in X_n$ such that $(x)\overline{A}^{i} = x$. 
 
 This means, since the map from $\spn{n}$ to the monoid of transformations of the set $\xn{n}$ by $D \mapsto \overline{D}_{1}$ is a homomorphism, that there is a state of $q_{0}$ of $\min\core(A^i)$ with a loop labelled $x|x$ based at $q_0$. This readily implies that for any power $A^{ki}$ of $A^i$ the state $q_0^k$ is in the core, since this is the unique state of $A^{ki}$ with loop labelled $x|x$. Therefore we may take $B = \min(A^{i}_{q_0})$.
 \end{proof}
 
 We have the following conjecture about the growth rates of elements of $\shn{n}$:
 
 \begin{conjecture}\label{Conjecture}
 Let $A \in \shn{n}$ be an element of infinite order, then the core growth rate of $A$ is exponential. Moreover for any $i \in \N$, $|\min \core(A^{i})|$ is greater than or equal to $i$.
 \end{conjecture}
 
 Notice that the examples considered above  all satisfy Conjecture~\ref{Conjecture}. A strategy for verifying this conjecture is to show that in reducing to the core we do not lose too many states. To this end we make the following definition:
 
 \begin{definition}
 Let $A$ be a finite synchronous transducer. Then we say has \emph{core distance} $k$ if there is a natural number $k$ such that for any $\Gamma \in X_n^{k}$ and any $q \in A$, $\pi_{A}(\Gamma, q)$ is a state of $\core(A)$. Let $\coredist(A)$ be the minimal $k$ such that $A$ has core distance $k$. If $A= \core(A)$ then $\coredist(A) = 0$.
 \end{definition}
 
 The lemma below explores how the function $\coredist$ behaves under taking products.
 
  \begin{lemma}\label{coredistunderproducts}
  Let $A, B \in \spn{n}$ and let $k_A$ and $k_B$ be  minimal so that $A$ is synchronizing at level $k_A$ and $B$ is synchronizing at level $k_B$. Then $\coredist(A*B) \le k_B$.
  \end{lemma} 
  \begin{proof}
   Indeed observe that given a state $U$ of $A$ such that the transition $U \stackrel{x|y}{\longrightarrow} V$ for $x,y \in X_n^{k_{B}}$ and $V$ a state of $A$ holds in $A$, then since $U$ is in the core of $A$ (as $A= \core(A)$) there is a word, $z$ of length $k_{A}$ such that there is a loop labelled $z|t'$ based at $U$. Let  $p$ be the state of $B$ forced by $y$. 
   
   Observe that since $A$ is synchronizing at level $k_{A}$, there is a path $V \stackrel{z|t}{\longrightarrow}{U} \stackrel{x|y}{\longrightarrow} V$. Therefore there is a loop labelled $z x|t y$ based at  $V$.  Therefore in $A*B$ there is a loop labelled $z t$ based at $Vp$, since the state of $B$ forced by $y$ is  $p$. Hence for any state $Uq$ of $A*B$, we read an $x$ into a state $Vp$ which is in $\core(A*B)$.
   
   \end{proof}
  
  We have as a corollary:
 
 \begin{lemma}\label{coredistunderpowers}
 Let $A \in \spn{n}$ be synchronizing at level $1$. Let $A^m$ represent the minimal transducer representing the core of $A^{m}$, then $\coredist(A^m*A) \le 1$.
 \end{lemma} 

Notice that by lemma \ref{core=power} there are elements $A \in \spn{n}$ for which $\coredist{A^m} = 0$ for all $m \in \mathbb{N}$. 

\begin{lemma}
Let $A \in \hn{n}$ by bi-synchronizing at level $k$. Then $\coredist(A ^m) \le \ceil{mk/2}$.
\end{lemma}
\begin{proof}
First notice that $A^{m} = A^{\floor{m/2}}* A^{\ceil{m/2}}$. Furthermore both $A^{\floor{m/2}}$and $A^{\ceil{m/2}}$ are bi-synchronizing at level $\ceil{m/2}$.

Let $U$ and $V$ be states respectively of $A^{\floor{m/2}}$ and $A^{\ceil{m/2}}$. Let $\Gamma \in X_n^{\ceil{m/2}}$. Suppose we have the transition:
\[
U \stackrel{\Gamma| \Delta}{\longrightarrow} U'.
\]

Since $A^{\floor{m/2}}$ is bi-synchronizing at level $\ceil{m/2}$, then the state of $A^{-\floor{m/2}}$ forced by $\Delta$ is $U'^{-1}$ (the state of $A^{-\floor{m/2}}$ corresponding to $U'$). Therefore there is a loop labelled $\Delta|\Gamma'$ based at $U'^{-1}$ in $A^{-\floor{m/2}}$, hence there is a loop labelled $\Gamma'|\Delta$ based at $U'$ in $A^{\floor{m/2}}$. 

Let $T'$ be the state of $A^{\ceil{m/2}}$ forced by $\Delta$, then $U'T'$ is in $\core(A^m)$. 

Hence we have shown that for any state $T$ of $A^{\ceil{m/2}}$  then the state $UT$ is at most $\ceil{m/2}$ steps from $\core(A^{m})$. Since $U$ was chosen arbitrarily this concludes the proof. 
\end{proof}

Lemma  \ref{core=power} once again shows that the lemma above is an over-estimate in some cases.

If we are able to obtain good bounds on the function $\coredist$ for a given transducer $A \in \hn{n}$ of infinite order, then it is possible to prove core exponential growth. In particular it is not hard to show that if there is an $M \in \mathbb{N}$ such that $\coredist(A^{m}) \le M$ for all $m \in \mathbb{N}$ then  $A$ has core exponential growth rate if it has infinite order.

We have  seen above that there are elements of $\T{P}_{n}$ which attain the maximum core growth rate possible. The proposition below establishes a lower bound for the core growth rate of those elements $A$ of $\T{H}_{n}$ of infinite order such that their graph $G_r$ of bad pairs possesses a loop for some $r \in \mathbb{N}$.

We have the following result:

\begin{proposition}
Let $A \in \T{H}_{n}$ be an element of infinite order, and suppose that the graph $G_r(A)$ of bad pairs of $A$ has a loop for some $r \in \mathbb{N}$. Then $A$ has at least core polynomial growth.
\end{proposition}
\begin{proof}
By Lemma 4.8 of \cite{OlukoyaOrder} and the definition of the graphs $G_{r}(A)$, the synchronizing level of $A$ grows linearly with powers of $A$ 

By the collapsing procedure see \cite{BCMNO}, a transducer with minimal synchronizing level $i$ must have at least $i$ states, since at each step of this procedure we must be able to perform a collapse.

Therefore we conclude that the core growth rate of $A$ is at least linear in powers of $A$.
\end{proof}

The following lemma controls the drop in the synchronizing level of  a sufficiently large strongly synchronizing automaton when multiplied by a level $1$ synchronizing transducer.

\begin{lemma}
	Let $A \in \pn{n}$ be a core, minimal transducer such that $|A| > n(n+1)$ let $B$ be any transducer synchronizing at level $1$, then $\min\core(AB)$ is synchronizing at level strictly greater than  1.
\end{lemma}
\begin{proof}
	For each $ i \in X_n$ let $\T{I}_{i} := \{ \pi_{A}(i,p) | p \in Q_A \}$.  Notice since $A$ is strongly synchronizing and core it is also strongly connected, therefore for all $p \in Q_A$ there is a set $\T{I}_{i}$ for some $i \in X_n$ such that $p \in \T{I}_{i}$. It now follows that $\cup_{i \in X_n}\T{I}_{i} = Q_A$.
	
	Now if $|\T{I}_{i}| < n +1$ for all $i$ then:
	\[
	|A| = |\cup_{i \in X_n}\T{I}_{i}| \le  \sum_{i=1}^{n}|\T{I}_{i}| < n*(n+1) < |A|
	\]
	
	which is a contradiction. Therefore there must be an $i \in X_n$ such that $|\T{I}_{i}| > n+1$. Fix such an $i \in X_n$.
	
	Now since $|\T{I}_{i}| > n+1$, there must be states $p_1', p_2', p_1, p_2 \in Q_A$ such that $p_1 \ne p_2$ and $p_1' \ne p_2'$ and such that the following transitions are valid:
	
	\[
	p_1' \stackrel{i|j}{\longrightarrow}{p_1} \qquad p_2' \stackrel{i|j}{\longrightarrow}{p_2}
	\] 
	for some $j \in X_n$. 
	
	Now observe that there are states $(p_1', q_1')$ and $(p_2', q_2')$ in the core of $AB$ where $q_1$ and $q_2$ are states of $B$. Let $\pi_B(j, q_1') = q_j$ and $\pi_B(j, q_2') = q_j$ (since $B$ is synchronizing at level 1).
	
	Therefore the following transitions are valid:
	
	\[
	(p_1', q_1') \stackrel{i| l_1}{\longrightarrow}(p_1, q_j) \qquad (p_2', q_2') \stackrel{i| l_1}{\longrightarrow}(p_2, q_j) 
	\]
	
	where $l_1 = \lambda_{B}(j, q_1')$ an $l_2 = \lambda_{B}(j, q_2')$ Now if $\min \core(AB)$ is synchronizing at level 1, then $(p_1, q_j)$ an $(p_2, q_j)$ would be $\omega$-equivalent, since $(p_1', q_1')$ and $(p_2', q_2')$ are states in the core of $AB$. However $(p_1, q_j) \wequal (p_2, q_j)$ implies that $p_1 \wequal p_2$, but by assumption $p_1$ and $p_2$ are distinct and $A$ is minimal and so $p_1 \wequal p_2$ is a contradiction.
	
	Therefore $\min \core (AB)$ is not synchronizing at level 1.
	
\end{proof}

\subsection*{Acknowledgements}
The author wishes to acknowledge support from EPSRC research grant EP/R032866/1 and Leverhulme Trust Research Project Grant RPG-2017-159.

\def\cprime{$'$}
\providecommand{\bysame}{\leavevmode\hbox to3em{\hrulefill}\thinspace}
\providecommand{\MR}{\relax\ifhmode\unskip\space\fi MR }
\providecommand{\MRhref}[2]{%
  \href{http://www.ams.org/mathscinet-getitem?mr=#1}{#2}
}
\providecommand{\href}[2]{#2}

%\bibliographystyle{amsplain}
%\bibliography{references} 

\end{document}